\documentclass{article}
 \usepackage{amsmath}
\usepackage{amssymb}
\usepackage{latexsym}
\usepackage{amsthm}
\usepackage{mathrsfs}
\usepackage{wasysym}
\usepackage{fancyhdr}
\usepackage{amsfonts}
\usepackage{amssymb}
\usepackage{epsfig}

\newtheorem{theorem}{Theorem}[section]

\newtheorem{lemma}[theorem]{Lemma}

\newtheorem{proposition}[theorem]{Proposition}

\newtheorem{conjecture}[theorem]{Conjecture}
\newtheorem{observation}[theorem]{Observation}
\theoremstyle{remark}

\newcommand{\scc}{\text{\upshape{scc}}}

\title{Circuit  covers of cubic signed graphs}
\author{Yezhou Wu\thanks{Ocean College, Zhejiang University, Hangzhou , Zhejiang 310027, P.R. China; Email: yezhouwu@zju.edu.cn. Partially supported by a grant of National Natural Science Foundation of China (No. 11501504) and a grant of Natural Science Foundation of Zhejiang Province (No. Y16A010005).}\; and
Dong Ye\thanks{Department of Mathematical Sciences and Center for Computational Sciences, Middle Tennessee State University, Murfreesboro, TN 37132. Email: dong.ye@mtsu.edu. Partially supported by Simons Foundation Grant \#396516.}}

\begin{document}

\maketitle

\begin{abstract}
A signed graph is a graph $G$ associated with a mapping
$\sigma: E(G)\to \{-1,+1\}$, denoted by $(G,\sigma)$. A {\em cycle} of $(G,\sigma)$ is a connected 2-regular subgraph. A cycle $C$ is {\em positive} if it has an even number of negative edges, and negative otherwise. A {\em circuit} of of a signed graph $(G,\sigma)$ is a positive cycle or a barbell consisting of two edge-disjoint negative cycles joined by a path. The definition of a circuit of signed graph comes from the signed-graphic matroid. A circuit cover of $(G,\sigma)$ is a family of circuits covering all edges of $(G,\sigma)$. A circuit cover with the smallest total length is called a shortest circuit cover of $(G,\sigma)$ and its length is denoted by $\scc(G,\sigma)$. Bouchet proved that
a signed graph with a circuit cover if and only if it is flow-admissible (i.e., has a nowhere-zero
integer flow). M\'a\v{c}ajov\'a et. al.
show that a 2-edge-connected signed graph $(G,\sigma)$ has $\scc(G,\sigma)\le 9 |E(G)|$ if it is flow-admissible.
This bound was improved recently by Cheng et. al.  to  $\scc(G,\sigma) \le 11|E(G)|/3$  for 2-edge-connected signed graphs with even negativeness, and particularly, $\scc(G,\sigma)\le 3|E(G)|+\epsilon(G,\sigma)/3$ for 2-edge-connected cubic signed graphs with even negativeness (where $\epsilon(G,\sigma)$ is the negativeness of $(G,\sigma)$). 
In this paper, we show that every 2-edge-connected cubic signed graph has $\scc(G,\sigma)\le 26|E(G)|/9$ if it is flow-admissible, and
$\scc(G,\sigma)\le 23|E(G)|/9$ if it has even negativeness.
\medskip


\end{abstract}

\section{Introduction}
Let $G$ be a graph. A cycle of $G$ is a connected 2-regular subgraph. A graph $G$ is {\em 2-edge-connected} if $G$ is connected and does not contain a {\em cutedge},
 whose deletion disconnects the graph $G$.  A {\em singed graph} $(G,\sigma)$ is a graph associated with a mapping $\sigma: E(G)\to \{-1,+1\}$, which is called a {\em signature} of $(G,\sigma)$. An edge $e$ is positive if $\sigma(e)=1$ and negative if $\sigma(e)=-1$. A graph is a special signed graph with only positive edges. Signed graphs are well-studied combinatorial structures due to their applications in combinatorics, geometry and matroid theory (cf. \cite{TZ} ).

A cycle $C$ of a signed graph $(G,\sigma)$ is {\em positive} if it contains an even number of negative edges, and {\em negative} otherwise. A {\em barbell} of a signed graph is a pair of edge-disjoint negative cycles joined by a path, which could have length zero. A {\em circuit} of a signed graph is a positive cycle or a barbell. The definition of circuit of signed graphs comes from the signed-graphic matroid (cf. \cite{TZ}). For a graph $G$, a circuit of $G$ is a cycle.  A {\em circuit cover} $\mathscr C$ of a signed graph $(G,\sigma)$ is a family of circuits which covers all edges of $G$. The length of a circuit cover $\mathscr C$ is defined as $\ell(\mathscr C)=\sum_{C\in \mathscr C} |E(C)|$.  A {\em shortest circuit cover} $\mathscr C$ of $(G,\sigma)$ is a circuit cover with the smallest length, i.e. $\ell(\mathscr C)$ is minimum, over all circuit covers of $(G,\sigma)$. The length of a shortest circuit cover of $(G,\sigma)$ is denoted by $\scc(G,\sigma)$.

The shortest circuit cover problem has been well-studied for graphs (cf. \cite{Zhang}) and matroids (cf. \cite{Go, PDS}). Thomassen \cite{CT} showed that for a given graph $G$, it is NP-complete to determine $\scc(G)$, which settled a problem proposed by Itai et. al. \cite{ILPR}.  Bermond, Jackson and Jaeger~\cite{BJJ}, independently Alon and Tarsi \cite{AT} obtained the following result, which was further generalized by Fan~\cite{Fan90} to 2-edge-connected graph with positive weights on edges.

\begin{theorem}[Bermond, Jackson and Jaeger \cite{BJJ}, Alon and Tarsi \cite{AT}]\label{thm:5/3}
Let $G$ be a 2-edge-connected graph. Then $\scc(G)\le 5|E(G)|/3$.
\end{theorem}

The bound in the above theorem was further improved to $44|E(G)|/27$ by Fan \cite{Fan94} for 2-edge-connected cubic graphs. For cubic graphs $G$ with a nowhere-zero 5-flow,  Jamshy, Raspaud and Tarsi~\cite{JRT} show that $\scc(G)\le 8|E(G)|/5$. With additional information on cycle or 2-factor structures, some upper bounds on shortest circuit cover of cubic graphs are obtained in \cite{CL, HZ, KKLNS}. In general, Alon and Tarsi made the following conjecture -- the Shortest Circuit Cover Conjecture.

\begin{conjecture}[Alon and Tarsi \cite{AT}]\label{conj:SCDC}
Every 2-edge-connected cubic graph has a shortest circuit cover with length at most $7 |E(G)|/5$.
\end{conjecture}

Jamshy and Tarsi~\cite{JT} proved that Conjecture~\ref{conj:SCDC} implies the well-known Circuit Double Cover Conjecture, proposed independently by Szekeres~\cite{GS} and Seymour~\cite{PDS}.

\begin{conjecture}[Szekeres \cite{GS} and Seymour \cite{PDS}]\label{conj:CDC}
Every 2-edge-connected graph has a family of circuits which covers every edge twice.
\end{conjecture}

By the splitting lemma of Fleischner (Lemma III.26 in \cite{HF}), it suffices to show that Conjecture~\ref{conj:CDC} holds for all 2-edge-connected cubic graphs. For 2-edge-connected cubic graphs, Conjecture~\ref{conj:CDC} is equivalent to another long-standing problem, the Strong Embedding Conjecture due to Haggard~\cite{Ha}, which says that every 2-connected graph has an embedding in a closed surface such that every face is an open disc and is bounded by a cycle, so-called a strong embedding. Based on the coloring-flow duality, the dual of a digraph embedded in an orientable surface is a  graph (or balanced signed graph) but the dual of a digraph embedded in a non-orientable surface is a signed graph (cf. \cite{DGMVZ}). By the duality, the dual of a digraph strongly embedded in a non-orientable surface is a signed graph with an even number of negative edges. It is interesting to ask: for a given 2-connected signed graph $(G,\sigma)$ with an even number of negative edges, is $(G,\sigma)$ a dual of some digraph strongly embedded in a non-orientable surface? If so, then $(G,\sigma)$ has a circuit double cover because every face boundary of $(G,\sigma)$ is a positive cycle. A weaker question is whether a 2-connected signed graph with an even number of negative edges has a circuit double cover or not? The answer to this question is negative, even for 3-connected cubic signed graph. The signed graph in Figure~\ref{fig:noCDC} has no circuit double cover.

\begin{figure}[!hbtp] \refstepcounter{figure}\label{fig:noCDC}
\begin{center}
\includegraphics[scale=1]{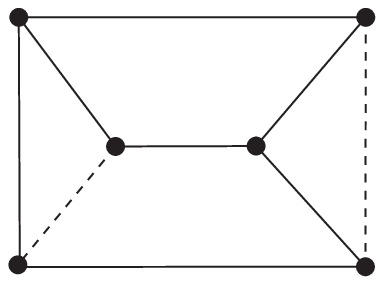}\\
{Figure~\ref{fig:noCDC}: A 3-connected signed graph without a circuit double cover \\
(solid edges are positive and dashed edges are negative).}
\end{center}
\end{figure}

It is natural to consider the shortest circuit cover problem for signed graphs. Let $(G,\sigma)$ be a signed graph with a circuit cover. Does $(G,\sigma)$ have a shortest circuit cover with length less than $2|E(G)|$, which follows directly if $(G,\sigma)$ has a circuit double cover? However, the above examples show that some signed graphs do not have a circuit double cover. The shortest circuit cover problem for signed graph have been studies by M\'a\v{c}ajov\'a et. al. \cite{mrrs} and Cheng et. al. \cite{CLLZ}. Before presenting their results, we need some terminologies.

The 2-edge-connectivity condition is sufficient for a graph to have a circuit cover, but it does not guarantee the existence of a circuit cover for a signed graph. A signed graph $(G,\sigma)$ is {\em flow-admissible} if $(G,\sigma)$ has a nowhere-zero flow. If $(G,\sigma)$ has a circuit cover, then every edge of $(G,\sigma)$ is contained by a circuit (a positive cycle or a barbell). Note that a positive cycle has a nowhere-zero 2-flow, but a barbell has a nowhere-zero 3-flow (cf. \cite{Bo}). So a signed graph with a circuit cover is flow-admissible. Bouchet \cite{Bo} proved that a signed graph $(G,\sigma)$ has a circuit cover property if and only if it is flow-admissible.  M\'a\v{c}ajov\'a et. al. \cite{mrrs} obtained the following result.

\begin{theorem}[M\'a\v{c}ajov\'a, Raspaud, Rollov\'a and \v{S}koviera, \cite{mrrs}]
Let $(G,\sigma)$ be a 2-edge-connected signed graph. If $(G,\sigma)$ is flow-admissible, then $\scc(G,\sigma)\le 9|E(G)|$.
\end{theorem}

The above result was improved recently by Cheng et. al. \cite{CLLZ} for 2-edge-connected signed graph $(G,\sigma)$ with even negativeness 
 (see Section 2 for definition of negtiveness) as follows.

\begin{theorem}[Cheng, Lu, Luo and Zhang \cite{CLLZ}]
Let $(G,\sigma)$ be a 2-edge-connected signed graph with even negativeness. Then $\scc(G,\sigma)\le 11|E(G)|/3$.
\end{theorem}

For 2-connected cubic signed graphs $(G,\sigma)$ with even negativeness, the above bound could be improved to  $\scc(G,\sigma)\le 3|E(G)|+\epsilon(G,\sigma)/3$  in terms of negativeness $\epsilon(G,\sigma)$ of $(G,\sigma)$ (see \cite{CLLZ}). In this paper, we consider the shortest circuit cover of cubic signed graphs and the following is our main result.

\begin{theorem}\label{thm:main}
Let $(G,\sigma)$ be a 2-edge-connected cubic signed graph. If $(G,\sigma)$ is flow-admissible, then $\scc(G,\sigma)\le 26|E(G)|/9$.
\end{theorem}

For 2-connected cubic signed graphs $(G,\sigma)$ with even negativeness, the bound in Theorem~\ref{thm:main}  can be improved to $\scc(G,\sigma)\le 23|E(G)|/9$ as shown in Theorem~\ref{thm:even}.

\section{Preliminaries }

Let $(G,\sigma)$ be a connected signed graph. If $H$ is a subgraph of $G$,  the {\em signed subgraph} of $(G,\sigma)$ consisting of edges in $H$ together with their signatures is denoted by $(H,\sigma)$. An edge-cut $S$ of $(G,\sigma)$ is a minimal set of edges whose removal disconnects the signed graph.  A {\em switch operation $\zeta$}
on $S$ is a mapping $\zeta: E(G)\to \{-1,1\}$ such that $\zeta (e)=-1$ if $e\in S$ and $\zeta(e)=1$ otherwise. Two signatures $\sigma$ and $\sigma'$ are {\em equivalent} if there exists an edge cut $S$ such that $\sigma (e)=\zeta (e) \cdot \sigma'(e)$ where $\zeta$ is the switch operation on $S$. For any edge-cut $S$, a cycle $D$ of $(G,\sigma)$ contains an even number of edges from $S$. So a circuit $C$ of $(G,\sigma)$ is also a circuit of $(G,\sigma')$ for any equivalent signature $\sigma'$ of $\sigma$. Therefore, we immediately have the following observation.

\begin{observation}\label{lem:equ}
Let $(G,\sigma)$ be a flow-admissible signed graph and $\sigma'$ be an equivalent signature of $\sigma$.
Then $\scc(G,\sigma)=\scc(G,\sigma').$
\end{observation}

The {\em negativeness} of a signed graph $(G,\sigma)$ is the smallest number
of negative edges over all equivalent signatures of $\sigma$, denoted by $\epsilon(G,\sigma)$. M\'a\v{c}ajov\'a and \v{S}koviera~\cite{MS} proved that a 2-edge-connected signed graph is flow-admissible if and only if $\epsilon(G,\sigma)\ne 1$. Combining it with Boucet's result \cite{Bo} that a signed graph with a circuit cover if and only if it is flow-admissible, the following observation holds.

\begin{observation}\label{ob:1-negative-edge}
Let $(G,\sigma)$ be a 2-edge-connected signed graph. Then $(G,\sigma)$ has a circuit cover if and only if $\epsilon(G,\sigma)\ne 1$.
\end{observation}

If $(G,\sigma)$ has the smallest number of negative edges, an edge cut $S$ has at most half number of negative edges. Otherwise, apply the switch operation on $S$ and the number of negative edges of $(G,\sigma)$ is reduced, contradicting that $(G,\sigma)$ has the smallest number of negative edges.

 \begin{observation}\label{ob:cut}
If a signed graph $(G,\sigma)$ with $\epsilon(G,\sigma)$ negative edges, then every edge cut $S$ contains at most
$|S|/2$ negative edges.
\end{observation}

A connected graph $H$ is called a {\em cycle-tree} if it has no vertices of degree-1 and all cycles of $H$ are edge-disjoint. If $H$ is a cycle-tree, then the graph obtained  from $H$ by contracting all edges in cycles is a tree. In other words, a cycle-tree can be obtained from a tree by blowing up all leaf vertices and some non-leaf vertices to edge disjoint cycles. A vertex $v$ of $H$ is a {\em cutvertex} if $H\backslash \{v\}$ has more components than $H$. A cutvertex $v$ is said to {\em separate} a graph $H$ into $H_1$ and $H_2$ if $H_1\cup H_2=H$ and $H_1\cap H_2=v$. Note that both $H_1$ and $H_2$ are connected since
$H$ is connected.
A cycle $D$ of $H$ is a {\em leaf-cycle} if $H$ has a vertex $v$ separating $D$ and $H\backslash (V(D)\backslash \{v\})$.

A {\em signed cycle-tree} $(H,\sigma)$ is a signed graph such that $H$ is a cycle-tree and {\bf every cycle of $(H,\sigma)$ contains at least one
 negative edge}.
Let $\mathcal F$ be a family of circuits of $(H,\sigma)$. A cycle $D$ of $(H,\sigma)$ is covered {\em $t$ times} ($t\ge 1$) by $\mathcal F$ if every edge of $D$ is covered by $\mathcal F$ and $\sum_{D'\in \mathcal F}|E(D)\cap E(D')|=t|E(D)|$, where $t$ is a rational number.

\begin{lemma}\label{lem:tree-circuit-cover}
Let $(H,\sigma)$ be a signed cycle-tree with an even number of negative cycles.
Then $(H,\sigma)$ has a family of circuits $\mathcal F$ which covers all leaf-cycles once and all
other cycles at most 3/2-times.
\end{lemma}
\begin{proof}
Let $(H,\sigma)$ be a counterexample with the smallest number of edges. First, we have
the following claim:\medskip

{\bf Claim:}  {\sl $(H,\sigma)$ does not contain a cutvertex $v$ which separates $H$ into two subgraphs $H_1$ and $H_2$ such that both $(H_1,\sigma)$ and $(H_2,\sigma)$ contain an even number of negative cycles.}
\medskip

{\em Proof of Claim:} suppose to the contrary that $(H,\sigma)$ does have a such vertex $v$. Since both $H_1$ and $H_2$ are connected, both of them contains a cycle-tree with an even number of negative cycles. We may assume that $H_i$ is a cycle-tree. (If $H_i$ is not a cycle-tree, its maximum connected subgraph $H_i'$ without vertices of degree 1 is a cycle-tree. Then use $H_i'$ instead.)
Furthermore, a cycle of $H$ is contained in either $H_1$ or $H_2$.

Since $(H,\sigma)$ is a counterexample to the lemma with minimum number of edges and $|E(H_i)|<|E(H)|$, both
$(H_1,\sigma)$ and $(H_2, \sigma)$ have a family of circuits covering leaf-cycle
exactly once and other cycles at most 3/2-times. Denote the two families of circuits by $\mathcal F_1$
and $\mathcal F_2$ respectively. As $H_1$ and $H_2$ are separated by $v$, it follows that $(H_1,\sigma)$ and $(H_2,\sigma)$ have no cycle in common. Because a leaf-cycle of $(H,\sigma)$ is either a leaf-cycle of $(H_1,\sigma)$ or $(H_2,\sigma)$, it follows that $\mathcal F=\mathcal F_1\cup \mathcal F_2$ is a family of
circuits of $(H,\sigma)$ which cover leaf-cycles once and other cycles at most 3/2-times,  contradicting that $(H,\sigma)$ is a counterexample. This completes the proof of Claim.\medskip

In the following, we may assume first that  $(H,\sigma)$ contains a positive cycle $C$. Since $H$ is a cycle-tree, every component of $H\backslash E(C)$ has exactly one vertex on $C$, which is a cutvertex. By Claim, every component of $H\backslash E(C)$ contains an odd number of negative cycles. So the totally number of components of $H\backslash E(C)$ is even because $(H,\sigma)$ contains an even number of negative cycles.  Denote these components by $P_1, Q_1,
P_2,Q_2,\cdots, P_k, Q_k$, which appear in clockwise order along the cycle $C$.

Let $S_i$ be the segments of $C$ joining $P_i$ and $Q_i$ for $i=1,...,k$, and
$R_i$ be the segments of $C$ joining $Q_i$ and $P_{i+1}$ for $i=1,...,k$ (subscripts modulo $k$). Then $C=\bigcup_{i=1}^k (S_i\cup R_i)$.
Without loss of generality, assume that
\begin{equation} \label{eq:c/2}
\sum_{i=1}^k|E(S_i)|\le |E(C)| /2 \le \sum_{i=1}^k|E(R_i)|.
\end{equation}
Note that each component $(P_i\cup S_i\cup Q_i, \sigma)$ ($i=1,...,k$) of $H\backslash E(\cup_{i}^k R_i)$ is a signed cycle-tree with an even number of negative cycles. Because $|E(P_i\cup S_i\cup Q_i)|<|E(H)|$ and $(H,\sigma)$ is a counterexample with the smallest number of edges,
the signed cycle-tree $(P_i\cup S_i\cup Q_i, \sigma)$  has a desired family of circuits $\mathcal F_i$. Let
\[\mathcal F:=(\bigcup_{i=1}^k \mathcal F_i)\cup \{C\}.\]
By (\ref{eq:c/2}), $C$ is covered by $\mathcal F$ at most 3/2-times. Note that every leaf-cycle of $(H,\sigma)$ is also a leaf-cycle of $(P_i\cup S_i\cup Q_i, \sigma)$ for some unique $i\in \{1,...,k\}$. Therefore, $\mathcal F$ is a desired family of circuits of $(H,\sigma)$, contradicting that $(H,\sigma)$ is a counterexample.
So $(H,\sigma)$ does not contain a positive cycle.

If $(H,\sigma)$ contains exactly two negative cycles, then $(H,\sigma)$ itself is a barbell, denoted by $B$. Then $\{B\}$ is a desired family of circuits. Hence assume that $(H,\sigma)$ has at least four negative cycles. Choose a negative cycle $D$ of $(H,\sigma)$ such that the number of components of $H\backslash E(D)$ is maximum over all cycles of $H$.
By Claim, every component has an odd number of negative cycles. Therefore, $H\backslash E(D)$ has an odd number of components which is at least three by the choice of $D$. By a similar argument as in the case
 when $D$ is positive, we can label these components  by $Q_0, P_1, Q_1,P_2, Q_2,..., P_k, Q_k$ ($k\ge 1$) in clockwise order along $D$ such that
\begin{equation}\label{eq:c/2-2}
\sum_{i=1}^k |E(S_i)|\le |E(D)|/2,
\end{equation}
where $S_i$ is a segment of $D$ joining $P_i$ and $Q_i$ for $i=1,...,k$.

Let $H_0=Q_0\cup D$ and $H_i=P_i\cup S_i\cup Q_i$ for $i=1,...,k$. Then each $(H_i,\sigma)$  ($i=0,...,k$) is a signed cycle-tree with an even number of negative cycles.
Since $k\ge 1$, $|V(H_i)|<|V(H)|$ for all $i\in \{0,...,k\}$. As $(H,\sigma)$ is a counterexample with smallest number of edges,
each $(H_i,\sigma)$ is not a counterexample and therefore  has a family of circuits $\mathcal F_i$ which covers all leaf-cycle of $H_i$ once and other cycle at most 3/2-times.  Let
\[\mathcal F=\bigcup_{i=0}^k \mathcal F_i.\]
Since $D$ is a leaf-cycle of $H_0$, it is covered by $\mathcal F_0$ once. By (\ref{eq:c/2-2}), all $\mathcal F_1,...,\mathcal F_k$ together cover at most half number edges of $D$. Therefore, $D$ is covered by $\mathcal F$ at most 3/2-times. Since any other cycle is covered by only one of $\mathcal F_i$'s,
it follows that $\mathcal F$ is a desired family of circuits of $(H,\sigma)$, a contradiction to that $(H,\sigma)$ is a counterexample. This completes the proof.
\end{proof}

\begin{theorem} \label{thm:tree-cover}
Let $(H,\sigma)$ be a signed cycle-tree with an even number of negative cycles.
Then $(H,\sigma)$ has a family of
circuits $\mathcal F$ covering all cycles with length
\[\ell (\mathcal F)\leq \frac 43 |E(H)|.\]
\end{theorem}

\begin{proof}
Use induction on the number of edges of $(H,\sigma)$. If $(H,\sigma)$ has no edges, then the theorem holds trivially by taking $\mathcal F=\emptyset$. So in the following, assume that the theorem holds for all signed cycle-trees with at most $|E(H)|-1$ edges.

First, assume that $(H,\sigma)$ contains a positive leaf-cycle $C$. Let $H'\subset H$ be a cycle-tree containing all cycles of $H$ except $C$. Then $(H',\sigma)$ has an even number of negative cycles. Since $|E(H')|<|E(H)|$, by inductive hypothesis, $(H', \sigma)$ has a family of circuits $\mathcal F'$ covering all cycles of $(H',\sigma)$ with length $\ell(\mathcal F')\le 4|E(H')|/3$. Then
$ \mathcal F = \mathcal F'\cup \{C\}$  is a family of circuits covering all cycles of $(H,\sigma)$ because a cycle of $H$ is either a cycle of $H'$ or $C$. The length of $\mathcal F$ is
\[ \ell(\mathcal F)= \ell (\mathcal F')+|E(C)| \le \frac 4 3 |E(H')| +|E(C)|\le \frac 4 3 (|E(H')|+|E(C)|) \le \frac 4 3 |E(H)| .\]
So $(H,\sigma)$ has a family of circuits $\mathcal F$ covering all cycles with length at most $4|E(H)|/3$.

In the following, assume that all leaf-cycles of $(H,\sigma)$ are negative.
Let $D_1, D_2, ..., D_k$ be all leaf-cycles. Let $l$ be the total length of non-leaf cycles of $H$. Since $H$ is an outerplanar graph, $H$ has an embedding in the plane such that all vertices of $H$ appear on the boundary of the infinite face.
Let $W$ be the closed walk bounding the infinite face. Then all vertices of a leaf-cycle $D_i$ appears as a consecutive segment in $W$.
Without loss of generality, assume that the leaf-cycles of $H$ appears in $W$ in the order $D_1, D_2,...,D_k,$.
Let $S_{i,i+1}$ be the segment of $W$ joining $D_i$ and $D_{i+1}$ (subscribes modulo $k$) such that all internal vertices of $S_{i,i+1}$ do not belong to any leaf-cycle of $(H,\sigma)$. Then $S_{i,i+1}$ is a path because $H$ does not have vertices of degree 1.
Let \[B_i=D_i\cup D_{i+1}\cup S_{i,i+1} \mbox{ for } i=1,2,..., k \mbox{ (subscribes modulo } k).\]
Then $B_i$ is a barbell for $i=1,...,k$. Let $\mathcal F_1=\{B_1, B_2,...,B_k\}$, which
covers all edges in non-leaf cycles exactly once and all other edges twice.
So $\ell(\mathcal  F_1)= 2|E(H)|-l$.

By Lemma~\ref{lem:tree-circuit-cover}, $(H,\sigma)$ has a  family of circuits $\mathcal F_2$ covering all cycles
with length $\ell(\mathcal F_2)\le |E(H)|+l/2$.
Let $\mathcal F$ be the family of circuits  with the smaller length between $\mathcal F_1$ and $\mathcal F_2$.  Then
\[\ell(F) \le \frac 1 3\big ( \ell (\mathcal F_1)+\ell (\mathcal F_2)+\ell( \mathcal F_2)\big)=\frac 1 3 \big ((2|E(H)|-l)+2 (|E(H)|+l/2  )\big )=\frac 4 3 |E(H)|.\]
 This completes the proof.
 \end{proof}

\section{Shortest circuit covers}

In this section, we consider the shortest circuit covers of cubic signed graphs. Let $(G,\sigma)$ be a 2-edge-connected signed
graph and let $E^-(G,\sigma):=\{e\ | \sigma(e)=-1\}$ and $E^+(G,\sigma):=\{e\ | \sigma(e)=1\}$. By Observation~\ref{lem:equ},
 we may always assume $(G,\sigma)$ has the smallest number of negative edges over all equivalent signatures of $\sigma$. In other words, $|E^-(G,\sigma)|=\epsilon(G,\sigma)$.  Let $G^+$ be the subgraph of $G$ induced by edges in $E^+(G,\sigma)$, i.e., $G^+=G\backslash  E^-(G,\sigma)$. By Observation~\ref{ob:cut}, for any edge-cut $S$, the following inequalities hold
\begin{equation}\label{eq:cut}
|E^-(G,\sigma)\cap S|\le |S|/2\le |E^+(G,\sigma)\cap S|.
\end{equation}
So $G^+$  is connected spanning subgraph of $G$.

\begin{lemma}\label{lem:cutedge}
Let $(G,\sigma)$ be a 2-edge-connected signed graph with $|E^-(G,\sigma)|=\epsilon(G,\sigma)$.
If $(G,\sigma)$ has a family of circuits $\mathcal F$ such that every negative edge $e$ is contained in a cycle of $\bigcup\limits_{C\in \mathcal F} C$,
then $\mathcal F$ covers all cutedges of $G^+=G\backslash  E^-(G,\sigma)$.
\end{lemma}

\begin{proof} Suppose to the contrary that $G^+$ has a cutedge $e$
which is not covered by any circuit in $\mathcal F$. Let $S$ be an edge-cut of $G$ such that $S\cap E^+=\{e\}$.
Since $|E^-(G,\sigma)|=\epsilon(G,\sigma)$, then  $|S|/2 \le |E^+(G,\sigma)\cap S|=1$ by (\ref{eq:cut}). The 2-edge-connectivity of $(G,\sigma)$ implies that $|S|=2$. Let $e'$ be the other edge in $S$. Then $e'\in E^-(G,\sigma)$. Note that $e'$ is contained by a cycle $D$ of $\bigcup\limits_{C\in \mathcal F} C$. Note that $|E(D)\cap S|$ is even. Therefore, the cycle $D$ contains $e$ too. So $e$ is covered by $\mathcal F$, a contradiction. This completes the proof.
\end{proof}

Let $(G,\sigma)$ be a 2-edge-connected flow-admissible cubic signed graph. In order to show that $(G,\sigma)$ has a small circuit cover, we need to find  a family of circuits with a suitable length to cover all negative edges and all bridges of $G^+$, and another family of circuits to cover the rest of edges.
By Theorem~\ref{thm:5/3}, there is a family of circuits of $G^+$ covering all edges of $G^+$ except these cutedges with length at most $5|E(G^+)|/3$. Hence, by Lemma~\ref{lem:cutedge}, it suffices to find a family of circuits $\mathcal F$ with a suitable length such that every edge of  $E^-(G,\sigma)$ is covered by a cycle of some circuit in $\mathcal F$.

Let $T$ be a spanning tree of $G^+$. Then $T$ is also a spanning tree of $G$ because $G^+$ is a spanning subgraph of $G$.
For any $ e\in E^-(G,\sigma)\subseteq E(G)\backslash E(T)$, let $D_e$ be the elementary cycle of $T\cup \{e\}$. Since $G$ is cubic, the symmetric difference of all cycles $D_e$, denoted by $\mathscr D$, consists of disjoint cycles. Let $Q$ consists of  all cycles of $\mathscr  D$ with negative edges.  Because a negative edge $e$ is contained by only $D_e$, $Q$ contains all negative edges of $(G,\sigma)$, i.e., $E^-(G,\sigma)\subseteq E(Q)$. Let $H$ be a minimal connected subgraph of $G$ such that $Q\subseteq H\subseteq Q\cup T$. By the minimality of $H$, $H$ has no vertices of degree 1 and any edge $e$ of $E(H)\backslash E(Q)$ is a cutedge. (Otherwise, $H\backslash \{e\}$ is still connected and satisfies $Q\subseteq H\cup\{e\} \subseteq Q\cup T$, a contradiction to the minimality of $H$.) So $H/E(Q)$ is a tree and hence $H$ is a cycle-tree. So $(H,\sigma)$ is a signed cycle-tree of $(G,\sigma)$ such that $E^-(G,\sigma)\subseteq E(H,\sigma)$.

Before proceed to prove our main result\,---\,Theorem~\ref{thm:main}, we show a better bound for 2-edge-connected cubic signed graphs with even negativeness.  By Obeservation~\ref{ob:1-negative-edge}, a 2-edge-connected signed graph with even negativeness always has a circuit cover.

\begin{theorem}\label{thm:even}
Let $(G,\sigma)$ be a 2-edge-connected cubic signed  graph with even negativeness. Then \[\scc (G,\sigma) <\frac{23} 9 |E(G)|.\]
\end{theorem}

\begin{proof} If $\epsilon(G,\sigma)=0$, then $(G,\sigma)$ is a graph and hence $\scc(G,\sigma)\le 5|E(G)|/3$ by Theorem~\ref{thm:5/3}. The theorem holds immediately. So in the following, assume that $\epsilon(G,\sigma)\ge 2$ and $|E^-(G,\sigma)|=\epsilon(G,\sigma)$ by Observation~\ref{lem:equ}.

Recall that $(G,\sigma)$ has a signed cycle-tree $(H,\sigma)$ such that $E^-(G,\sigma)\subseteq E(H,\sigma)$.
Since $\epsilon(G,\sigma)$ is even, it follows that $(G,\sigma)$ has an even number of negative cycles. By Theorem~\ref{thm:tree-cover}, $(H,\sigma)$ has a family of circuits $\mathcal F_1$ which covers all cycles of $(H,\sigma)$ and hence covers all negative edges of $(G,\sigma)$ with length
\begin{equation}\label{eq:even}
\ell(\mathcal F_1)\le \frac 4 3 |E(H)|\le \frac 4 3 (|V(G)|-1+|E^-(G,\sigma)|)=\frac 8 9 |E(G)|+\frac 4 3 |E^-(G,\sigma)|-\frac 4 3.
\end{equation}

By Lemma~\ref{lem:cutedge}, $\mathcal F_1$ covers all cutedges of $G^+$. Deleting all cutedges from $G^+$, every component of the resulting graph is 2-edge-connected. By Theorem~\ref{thm:5/3}, all shortest circuit covers of these components together form a family of circuits $\mathcal F_2$ of $G^+$, which covers all edges of $G^+$ except cutedges with length
\[\ell (\mathcal F_2)\le \frac 5 3 |E(G^+)|=\frac 5 3 (|E(G)|-|E^-(G,\sigma)|).\]
So $\mathcal F=\mathcal F_1\cup \mathcal F_2$ is a circuit cover of $(G,\sigma)$
with length
\begin{equation*}\label{EQ:even}
\ell(\mathcal F)=\ell(\mathcal F_1)+\ell (\mathcal F_2)\le   ( \frac 8 9 |E(G)|+\frac 4 3 |E^-(G,\sigma)|-\frac 4 3)+ \frac 5 3 (|E(G)|-|E^-(G,\sigma)|)\le \frac{23} 9 |E(G)|-\frac 1 3 |E^-(G,\sigma)|-\frac 4 3 .
\end{equation*}
It follows that $\scc(G,\sigma)< 23|E(G)|/9$. So the theorem holds.
\end{proof}

In the following, we consider signed cubic graphs $(G,\sigma)$ with odd negativeness, i.e., $\epsilon(G,\sigma)$ is odd.
The {\it signed-girth} of a signed graph $(G,\sigma)$ is length of a shortest circuit containing negative edges, denoted by $g_s(G,\sigma)$.  Before proceed to prove our main result, we need some technical lemmas.

\begin{lemma}\label{lem:star}
Let $(G,\sigma)$ be a signed cubic graph, and $(N,\sigma)$ be a signed cycle-tree of $(G,\sigma)$. If $g_s(G,\sigma)\ge |E(G)|/3 +2$, then:

(1) $(G,\sigma)$ does not contain two disjoint circuits both containing negative edges;

(2) $(N,\sigma)$  has at most three leaf-cycles and at most one non-leaf cycle. Furthermore, if it has a non-leaf cycle, then all leaf-cycles are negative.
\end{lemma}
\begin{proof}
If $(G,\sigma)$ has only one circuit, the lemma holds trivially. So assume that $(G,\sigma)$ has at least two distinct circuits. Let $C_1$ and $C_2$ be two distinct circuits. If $V(C_1)\cap V(C_2)=\emptyset$, then  $|V(C_1)|+|V(C_2)|\le |V(G)|$.
Note that $E(C_i)\leq |V(C_i)|+1$ ($i=1,2$) and equality holds if and only if $C_i$ is a barbell.
Since $G$ is cubic, $|V(G)|=2|E(G)|/3$.
It follows that
\[|E(C_1)|+|E(C_2)|\leq |V(C_1)|+|V(C_2)|+2\leq |V(G)|+2=\frac 2 3 |E(G)|+2.\]
Without loss of generality, assume that $|E(C_1)|\leq |E(C_2)|$. Hence $|E(C_1)|\le  |E(G)|/{3}+1$, contradicting $g_s(G,\sigma)\ge |E(G)|/3 +2$.  This completes the proof of (1).

Since $(G,\sigma)$ does not contain two disjoint circuits, every signed cycle-tree $(N,\sigma)$ of $(G,\sigma)$ does not contain two disjoint circuits neither. Hence $(N,\sigma)$ has at most three leaf-cycles.

If $(N,\sigma)$ has two non-leaf cycles $D_1$ and $D_2$, then there is a leaf cycle $D_i'$ is connected to $D_i$ by a path $P_i$ for $i=1$ and $2$ such that $P_1\cap P_2=\emptyset$. Then $D_i'\cup P_i\cup D_i$ contains a circuit for both $i=1$ and 2, contradicting $(N,\sigma)$ does not have two disjoint circuits. So $(N,\sigma)$ has at most one non-leaf cycle.

If $(N,\sigma)$ has exactly one non-leaf cycle $D_1$,  then $D_1$ is connected to at least two leaf-cycles. If one of the leaf-cycles $C$ is positive, the other leaf-cycles together with $D_1$ contains a circuit disjoint from $C$, a contradiction. This completes the proof.
 \end{proof}

\begin{lemma}\label{lem:bigcircuit}
Let $(G,\sigma)$ be a 2-edge-connected cubic signed graph with $\epsilon(G,\sigma)\ge 3$ negative edges, and $G^+$ be the subgraph induced by positive edges. If $g_s(G,\sigma)\ge |E(G)|/3+2$, then
$(G,\sigma)$ has a family of circuits $\mathcal F$ covering all negative edges of $(G,\sigma)$ and all cutedges of $G^+$ such that 
\[\ell(\mathcal F)<\frac{11} 9 |E(G)|+\frac 5 3 \epsilon(G,\sigma).\]
\end{lemma}
\begin{proof}Let $(G,\sigma)$ be a cubic signed graph with $g_s(G,\sigma)\ge |E(G)|/3+2$.
If $\epsilon(G,\sigma)$ is even, the lemma follows from (\ref{eq:even}). So assume that $\epsilon(G,\sigma)$ is odd. Suppose to the contrary that $(G,\sigma)$ is a counterexample.

Since $\epsilon(G,\sigma)\ge 3$, $(G,\sigma)$ has a circuit cover by Observation~\ref{ob:1-negative-edge}. So $(G,\sigma)$ has a circuit $C$ containing negative edges. The circuit $C$ has a negative edge in a cycle.\medskip

{\bf Claim~1.} {\sl  For any negative edge $e$ contained in a cycle of some circuit, the signed graph $(G,\sigma)$ has a signed cycle-tree $(H,\sigma)$ which has all negative edges in cycles and $E^-(H,\sigma)=E^-(G,\sigma)\backslash \{e\}$.}

\medskip
{\em Proof of Claim~1.} Let $T$ be a spanning tree of $G^+$. For any $e'\in E^-(G,\sigma)\backslash \{e\}$, let $D_{e'}$ be the elementary cycle of $T\cup \{e'\}$. The symmetric difference $\mathscr D_e=\oplus_{e'\in E^-(G,\sigma)\backslash\{e\}} D_{e'}$ consists of disjoint cycles because $G$ is cubic. Let $Q_e$ consist of all cycles of $\mathscr D_e$ containing at least one negative edge.
Let  $H$ be a minimal connected subgraph satisfying $Q_e\subseteq H\subseteq Q_e\cup T$. By the minimality of $H$, we can conclude that $(H,\sigma)$ is a signed cycle-tree of $(G,\sigma)$ such that every edge in $E^-(G,\sigma)\backslash\{e\}$ is contained by a cycle of $(H,\sigma)$. Note that $e\notin E(H,\sigma)$.  So $E^-(H,\sigma)=E^-(G,\sigma)\backslash \{e\}$. This completes the proof of Claim~1.
\medskip

For any negative edge $e$ contained by a cycle of some circuit, among all such signed cycle-trees with property in Claim~1, choose a signed cycle-tree $(H_e,\sigma)$ with the smallest number of cycles. Since $\epsilon(G,\sigma)-1$ is even, it follows that $(H_e,\sigma)$ has an even number of negative cycles.

\medskip

{\bf Claim~2.} {\sl The signed cycle-tree $(H_e,\sigma)$ is a circuit. }\medskip

{\em Proof of Claim~2.} Suppose on the contrary that $(H_e,\sigma)$ is not a circuit. Then it has a non-leaf cycle $D_0$. Let $D_1,...,D_k$ be all leaf-cycles of $(H_e,\sigma)$.  Since $g_s(G,\sigma)\ge |E(G)|/3+2$, by (ii) of Lemma~\ref{lem:star},
$D_0$ is the only non-leaf cycle of $(H_e,\sigma)$, and $D_1,...,D_k$ are negative cycles where $2\le k\le 3$. Further, $(H_e,\sigma)$ has $k+1$ cycles.

Since $G$ is 2-edge-connected and cubic, there are two disjoint paths $P_1$ and $P_2$ from $D_1$ to $D_0$. Since $(G,\sigma)$ does not contain two disjoint circuits, for both $i=1$ and $2$, we have $P_i\cap D_t=\emptyset$ where $t=2$ or $k$. Let $v_1$ and $v_2$ be two endvertices of $P_1$, and $u_1$ and $u_2$ be two endvertices of $P_2$ such that $v_1,v_2\in V(D_1)$ and $u_1,u_2\in V(D_0)$. The two vertices $u_1$ and $u_2$ separate $D_0$ into two internally disjoint segements $S_1$ and $S_2$. Without loss of generality, assume $|E(S_1)|\le |E(S_2)|$.  Then $D_1\cup P_1\cup P_2\cup S_1$ has a positive cycle, denoted by $C_1$. If $C_1$ does not contain a negative edge, then both $S_1$ and $C_1\cap D_1$ do not contain a negative edges. So deleting all internal vertices of $S_1$ and $C_1\cap D_1$ from $H_e\cup (P_1\cup P_2)$ results in a signed cycle tree with $k$ cycles, contradicting that $(H_e,\sigma)$ has the smallest number of cycles. Hence $C_1$ is a positive cycle with negative edges, and  $|E(C_1)|\le (|E(D_1)|-1)+|E(P_1)|+|E(P_2)|+|E(S_1)|$.

Similary, there are two disjoint paths $P_1'$ and $P_2'$ from $D_2$ to $D_0$. Let $S_1'$ be the segment with smaller length of $D_0$ separated by two endvertices of $P_1'$ and $P_2'$.  And $D_2\cup P_1'\cup P_2'\cup S_1'$ contains a positive cycle $C_2$ wih negative edges. Since $(G,\sigma)$ does not contain two disjoint circuit, it follows that for both $i=1$ and 2, $P_i'\cap (P_1\cup P_2\cup D_1)=\emptyset$ and $P_i'\cap D_3=\emptyset$ if $k=3$.
So
\[\begin{aligned}
|E(C_1)|+|E(C_2)| &\le (|E(D_1)|-1)+\sum_{i=1}^2 |E(P_i)|+|E(S_1)|+(|E(D_2)|-1)+\sum_{i=1}^2|E(P_i')|+|E(S_1')|\\
&\le |E(D_1)|+|E(D_2)|+|E(D_0)|+|E(P_1\cup P_2)|+|E(P_1'\cup P_2')|-2\\
&= |V(D_1\cup D_2\cup D_0\cup P_1\cup P_2\cup P_1'\cup P_2')|+4-2\\
&\le |V(G)|+2\\
&\le \frac 2 3 |E(G)|+2.
\end{aligned}\]
Without loss of generality, assume $|E(C_1)|\le |E(C_2)|$. Hence $|E(C_1)|\le |E(G)|/3+1$, contradicting $g_s(G,\sigma)\ge |E(G)|/3+2$. This completes the proof of Claim~2.
\medskip

By Claim~2, in the following, for any negative edge $e$ contained in a cycle of some circuit,t $(H_e,\sigma)$ is a circuit. In other words, $(H_e,\sigma)$ is a positive cycle or a barbell. 
\medskip

{\bf Claim 3.} {\sl Let $C$ be a positive cycle with negative edges or the union of two disjoint negative cycles.
Then} \[|E(C)|\geq\frac{4} 9 |E(G)|+6.\]

{\em Proof of Claim 3.} If $C$ is the union of two negative cycles $D_1$ and $D_2$, then there are two disjoint paths $P$ and $P'$ joining $D_1$ and $D_2$ since $G$ is 2-edge-connected and cubic. For the case that $C$ is a positive cycle, let $P=P'=\emptyset$.

Let $e$ be a negative edge in a cycle of the circuit $C$. Note that $(H_e,\sigma)$ is a circuit.
Then both $\mathcal F_1=\{C\cup P\}\cup \{H_e\}$ and $\mathcal F_2=\{C\cup P'\}\cup \{H_e\}$ are two families of circuits covering all edges in $E^-(G,\sigma)$. Since every negative edge is contained either in a cycle of $C$ or a cycle of $(H_e,\sigma)$, by Lemma~\ref{lem:cutedge}, both $\mathcal F_1$ and $\mathcal F_2$ cover all cutedges of $G^+$.  Since $(G,\sigma)$ is a countexample, both $\mathcal F_1$ and $\mathcal F_2$ have length at least $11|E(G)|/9+5\epsilon(G,\sigma)/3$. So
\[\ell(\mathcal F_1)+\ell(\mathcal F_2)=2|E(C)|+|E(P)|+|E(P')|+2|E(C')|\geq \frac{22} 9 |E(G)|+\frac{10} 3 |\epsilon(G,\sigma)|.\]

 Since $C\cup P\cup P'$ is a connected subgraph of $G$ with at most four vertices of degree 3,
 it follows that $|E(C)|+|E(P)|+|E(P')|\leq |V(G)|+2$.
  Note that the circuit $H_e$ has at most two cycles and hence has at most two vertices of degree 3. Therefore,
 $|E(H_e)|\leq |V(G)|+1$.
It follows that
\[\begin{aligned}
|E(C)|&\geq  \frac{22} 9 |E(G)|+\frac{10} 3 |\epsilon(G,\sigma)|-(|V(G)|+2)-2(|V(G)|+1)\\
&=\frac{22} 9 |E(G)|+\frac{10} 3 |\epsilon(G,\sigma)|-2|E(G)|-4 \\
&\ge  \frac{4} 9 |E(G)|+6.
\end{aligned}\]
This completes the proof of Claim~3. \medskip

Since $(G,\sigma)$ contains at least three negative edges, let $e_i$ ($i=1,2,3$) be negative edges of $(G,\sigma)$ and $D_{e_i}$ be the elementary cycle $T\cup \{e_i\}$. Let $C_{ij} \subseteq D_{e_i}\oplus D_{e_j}$ be either a positive cycle or the union of two disjoint negative cycles, which contains both $e_i$ and $e_j$. 

By Claim~3,  we have
\begin{equation}\label{EQ:2}
|E(C_{12})|+|E(C_{13})|+|E(C_{23})|\geq\frac{4} 3 |E(G)|+18.
\end{equation}
On the other hand,  since $C_{23}\subseteq D_{e_2}\oplus D_{e_3}= (D_{e_1}\oplus D_{e_3})\oplus (D_{e_1}\oplus D_{e_2})=C_{12}\oplus C_{23}$,
it follows that $\{C_{12}, C_{13}, C_{23}\}$ covers each edge of $T\cup\{e_1,e_2,e_3\}$ at
most twice. Therefore,
$$|E(C_{12})|+|E(C_{13})|+|E(C_{23})|\leq 2|E(T\cup\{e_1,e_2,e_3\})|=2(|V(G)|+2)=\frac 4 3 |E(G)|+4,$$
a contradiction to (\ref{EQ:2}). This completes the proof of the lemma.
\end{proof}

Now we are going to prove the main result. Recall our main result here.

\medskip

\noindent{\bf Theorem~\ref{thm:main}.}
{\it Let $(G,\sigma)$ be a 2-connected cubic signed graph. If $(G,\sigma)$ is flow-admissible, then }
\[\scc(G,\sigma)< \frac{26} 9 |E(G)|.\]

\begin{proof} Let $(G,\sigma)$ be a 2-edge-connected flow-admissible cubic signed graph. If $\epsilon(G,\sigma)$ is even, the theorem follows from Theorem~\ref{thm:even}. So in the following, we always assume that $\epsilon(G,\sigma)$ is odd. By Observations~\ref{lem:equ} and \ref{ob:1-negative-edge}, we further assume that $|E^-(G, \sigma)|=\epsilon (G,\sigma)\ge 3$.

Let $G^+=G\backslash E^-(G,\sigma)$. By Theorem~\ref{thm:5/3}, $G^+$ has a family of circuits $\mathcal F_1$ covering all edges of $G^+$ except cutedges with length 
\[\ell(\mathcal F_2)\le \frac 5 3 |E(G^+)|=\frac 5 3 (|E(G)|-|E^-(G,\sigma)|=\frac 5 3 (|E(G)|- \epsilon(G,\sigma)).\]

If the signed-girth of $(G,\sigma)$ satisfies $g_s(G,\sigma)\ge |E(G)|/3+2$, then by Lemma~\ref{lem:bigcircuit}, $(G,\sigma)$ has a family of circuits $\mathcal F_2$ covering edges in $E^-(G,\sigma)$ and all cutedges of $G^+=G\backslash E^-(G,\sigma)$ with length
\[\ell(\mathcal F_2)< \frac{11} 9 |E(G)|+\frac 5 3 \epsilon(G,\sigma).\]
So $\mathcal F=\mathcal F_1\cup \mathcal F_2$ is a circuit cover of $(G,\sigma)$ with length 
\[\ell(\mathcal F)=\ell(\mathcal F_1)+\ell(\mathcal F_2)<\frac 5 3 (|E(G)|- \epsilon(G,\sigma))+ \frac{11} 9 |E(G)|+\frac 5 3 \epsilon(G,\sigma)=\frac{26} 9 |E(G)|.\]
So the theorem holds for all signed graphs with $g_s(G,\sigma)\ge |E(G)|/3+2$.

In the following, assume that $(G,\sigma)$ has a circuit $C$ with length at most $|E(G)|/3+1$. 
Let $e$ be a negative edge contained in a cycle of $C$, and let $(H_e,\sigma)$ be a signed cycle-tree of $(G,\sigma)$ containing all negative edges in $E^-(G,\sigma)\backslash \{e\}$ in cycles of $(H_e,\sigma)$. (Note that, such signed cycle-trees exists as shown in Claim~1 in Lemma~\ref{lem:bigcircuit}). By Theorem~\ref{thm:tree-cover}, $(H_e,\sigma)$ has a family of circuits $\mathcal F_2$ covering all cycles of $(H_e,\sigma)$ with length
\[\ell(\mathcal F_2)\le \frac 4 3 |E(H_e)|\le \frac 4 3 (|V(G)|-1 +|E^-(G,\sigma)\backslash \{e\}|)= \frac 8 9 |E(G)|+\frac 4 3 \epsilon(G,\sigma)- \frac 8 3.\]
So $\mathcal F_2\cup \{C\}$ covers all negative edges of $(G,\sigma)$ and every negative edge is contained by a cycle of some circuit of $\mathcal F_2\cup \{C\}$. Hence $\mathcal F_2\cup \{C\}$ covers all negative edges of $(G,\sigma)$ and all cutedges of $G^+$ by Lemma~\ref{lem:cutedge}.

Note that $G^+$ has a family of circuits $\mathcal F_1$ covering all edges of $G^+$ except cutedges.
So $\mathcal F=\mathcal F_1\cup \mathcal F_2\cup \{C\}$ is a circuit cover of $(G,\sigma)$ with length
\[\begin{aligned} \ell(\mathcal F)&=\ell(\mathcal F_1)+\ell(\mathcal F_2)+|E(C)|\\
&\le \frac 5 3 (|E(G)|-\epsilon(G,\sigma))+\frac 8 9 |E(G)|+\frac 4 3 \epsilon(G,\sigma)- \frac 8 3+\frac 1 3 |E(G)|+1\\
&\le \frac{26} 9|E(G)|-\frac 1 3 \epsilon(G,\sigma)-\frac 5 3\\
&< \frac{26} 9|E(G)|.
\end{aligned}\]
This completes the proof of Theorem~\ref{thm:main}.
\end{proof}

\section{Concluding remarks}

A 2-edge-connected signed graph $(G,\sigma)$ with a circuit cover may not have a circuit double cover. In the following, we construct infinitly many 2-edge-connected signed graphs $(G,\sigma)$ with even negativeness but without circuit double cover properties.

\begin{proposition}
Let $(G,\sigma)$ be a cubic signed graph with a circuit double cover $\mathcal F$. If $v$ is a vertex of degree-3 in a barbell $B\in \mathcal F$, then $v$ is a vertex of degree-3 in another barbell $B'\in \mathcal F$.
\end{proposition}
\begin{proof} Since $(G,\sigma)$ is cubic, there are exactly three edges $e_1, e_2, e_3$ incident with $v$.
Since $v$ is a vertex of degree-3 in $B$, $e_1,e_2$ and $e_3$ are covered once by $B$. So $e_1, e_2$ and $e_3$ are covered once by $\mathcal F\backslash \{B\}$. Hence, $e_1, e_2$ and $e_3$ belong to exactly one circuit in $\mathcal F$, which must be a barbell $B'$. So $v$ is a vertex of degree-3 in $B'$.
\end{proof}

\begin{figure}[!hbtp] \refstepcounter{figure}\label{fig:2noCDC}
\begin{center}
\includegraphics[scale=1]{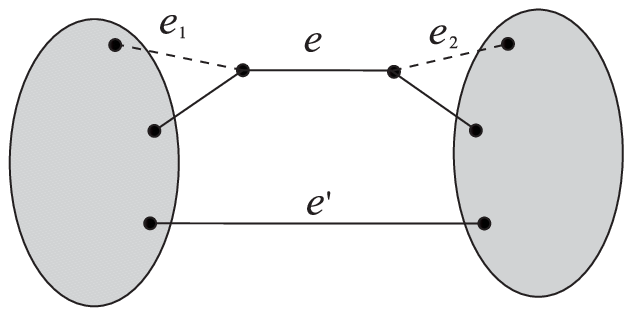}\\
{Figure~\ref{fig:2noCDC}: Infinitly many 2-connected signed graphs without a circuit double cover \\
(solid edges are positive and dashed edges are negative).}
\end{center}
\end{figure}

Let $G$ be a 2-connected cubic graph and $S=\{e,e'\}$ be a two edge-cut of $G$. Assume $e=uv$ and let $e_1$ incident with $u$ and $e_2$ incdient with $v$.  The signed graph $(G,\sigma)$ is obatined from $G$ by assigning -1 to both $e_1$ and $e_2$, and assigning 1 to all other edges. Suppose on the contrary that $(G,\sigma)$ have a circuit cover $\mathcal F$. If $\mathcal F$ has a barbell $B$, then $B\cap S\ne \emptyset$ since $e_1$ and $e_2$ belong two different cycles of $B$. We may assume that $e\in B$ (a similar argument works for $e'\in B$). 
Then $e$ is the path of $B$ joining the two cycles of $B$.
Hence both $u$ and $v$ are vertices of degree 3 in $B$. Then $v$ is a vertex of degree 3 in another barbell $B'$ in $\mathcal F$ by the above proposition. It follows that $e'$ can not be covered by any circuit of $(G,\sigma)$. So $\mathcal F$ does not have any barbell. Hence $e_1$ and $e_2$ are contained by two positive cycles $C_1$ and $C_2$ of $\mathcal F$. Then both $C_1$ and $C_2$ contain $S$. It follows that the third edge incident with $u$ or
$v$ different from $e_1, e_2$ and $e$ can not be covered by circuits in $\mathcal F$.  Hence $(G,\sigma)$ is a counterexample. This construction works for all cubic graphs with 2-edge-cut. Hence there are infinitly many 2-connected cubic signed graphs with a circuit cover but having no circuit double covers.

The example in Figure~\ref{fig:noCDC} shows that a 3-connected cubic signed graph with even negativeness may not have a circuit double cover.
By above proposition, any circuit double cover of the signed graph does not have a barbell. Because a circuit containing the two negative edges of the signed graph in Figure~\ref{fig:noCDC}  has length either 5 or 6, a counting of lengths of circuits shows that the signed graph has no circuit double covers.

As many 2-edge-connected signed graphs have no circuit double covers,  it is interesting to ask, is there an integer $k$ such that every 2-connected flow-admissible signed graph $(G,\sigma)$ has a circuit $k$-cover?


\end{document}